\documentclass[12pt]{article}

\usepackage[margin=1in]{geometry}  
\usepackage{graphicx}              
\usepackage{amsmath}               
\usepackage{amsfonts}              
\usepackage{amsthm}                

\newtheorem{thm}{Theorem}
\newtheorem{lem}{Lemma}[section]
\newtheorem{prop}[lem]{Proposition}
\newtheorem{cor}[lem]{Corollary}
\newtheorem{rmk}[lem]{Remark}

\newcommand{\RR}{\mathbb{R}}
\newcommand{\NN}{\mathbb{N}}
\newcommand{\be}{\begin{equation}}
\newcommand{\ee}{\end{equation}}
\newcommand{\bij}[1]{\sigma_#1:A_#1\ra\sn{r}\bs\{i(#1)\}}
\newcommand{\ra}{\rightarrow}
\newcommand{\s}{\subset}
\newcommand{\si}{\sigma}
\newcommand{\bs}{\backslash}
\newcommand{\lil}[2]{#1_1,\ldots,#1_#2}
\newcommand{\lu}[2]{#1^1,\ldots,#1^#2}
\newcommand{\suml}[2]{#1_1+\cdots+#1_#2}
\newcommand{\sn}[1]{\{1,\ldots,#1\}}

\begin{document}

\title{On Tverberg partitions}
\author{Moshe J. White \thanks{Research supported by ERC advanced grant 320924} \\
Institute of Mathematics,\\
Hebrew University, Jerusalem, Israel}
\date{}
\maketitle

\begin{abstract}
A theorem of Tverberg from 1966 asserts that every set $X\s\RR^d$ of $n=T(d,r)=(d+1)(r-1)+1$ points can be partitioned into $r$ pairwise disjoint subsets, whose convex hulls have a point in common. Thus every such partition induces an integer partition of $n$ into $r$ parts (that is, $r$ integers $\lil{a}{r}$ satisfying $n=\suml{a}{r}$), in which the parts $a_i$ correspond to the number of points in every subset. In this paper, we prove that for any partition of $n$ where the parts satisfy $a_i\leq d+1$ for all $i=1,\ldots,r$, there exists a set $X\s\RR$ of $n$ points, such that every Tverberg partition of $X$ induces the same partition on $n$, given by the parts $\lil{a}{r}$.
\end{abstract}

\section{Introduction}

Tverberg's theorem \cite{tverberg} from 1966 asserts the following: For any two integers $d,r$ define $n=T(d,r)=(d+1)(r-1)+1$. Then every $X\s\RR^d$ of $n$ points can be partitioned into $r$ disjoint subsets, $\lil{X}{r}$, so that $\cap_{i=1}^r conv(X_i)\neq\emptyset$. Such a partition is known as a Tverberg partition of $X$, and the points in $\cap_{i=1}^r conv(X_i)$ are known as Tverberg points of $X$.

In this paper, we prove the following theorem:

\begin{thm}
Given integers $d,r$ and $n=T(d,r)=(d+1)(r-1)+1$, let $\lil{a}{r}$ be integers satisfying $a_i\le d+1$ and $\sum_{i=1}^r a_i=n$ . Then there exists a set $X\s\RR^d$ of $n$ points, such that for any $\lil{X}{r}$ that form a Tverberg partition of $X$, the cardinalities of $\lil{X}{r}$ are a permutation of the integers $\lil{a}{r}$.
\end{thm}

\begin{rmk}
The requirement $a_i\leq d+1$ for all $i=1,\ldots,r$ is necessary:\\
Suppose $\lil{X}{r}$ is a Tverberg partition of $X$ with some Tverberg point $p$, and $|X_1|>d+1$ (without loss of generality). By Caratheodory's theorem, there exists $x\in X_1$ so that $p\in conv(X_1\bs\{x\})$. Therefore if we move the point $x$ from $X_1$ to the smallest subset among $X_2,\ldots,X_r$, we obtain a new Tverberg partition of $X$, which induces a different integer partition on $n$.
\end{rmk}

We will construct the family of sets $X$ that will be used in this proof, follow by proving some properties that apply to such sets, prove the theorem, and conclude by showing that the number of Tverberg partitions for these sets is always $[(r-1)!]^d$.

\section{Construction}
Assume $d,r,n \in\NN$ satisfy $n=(d+1)(r-1)+1$. We construct the set $X$ as a union of $d+1$ disjoint sets, $A,A_1,\ldots,A_d$, defined as follows:
Denote by $\lu{e}{d}$ the standard basis in $\RR^d$. For $1\le j\le d$, $$A_j=\{e^j,2e^j,\ldots,(r-1)e^j\}.$$

Let $\lil{a}{r}$ be integers as in Theorem 1. Note that:
$$\sum_{i=1}^r (d+1-a_i)=r(d+1)-\sum_{i=1}^r a_i=r(d+1)-n=d.$$
Therefore for each $1\le j\le d$ we can choose $1\le i(j)\le r$, so that  for all $1\le i\le r$, \be \#\{1\le j\le d: i(j)=i\}=d+1-a_i. \ee

For every $1\leq i\leq r$ and every $1\leq j\leq d$, define:
$$x^i_j=
\begin{cases} 0 & \text{if $i=i(j)$} \\
-i & \text{otherwise.}
\end{cases}$$
The set $A$ is now defined as $A=\{\lu{x}{r}\}$, where $x^i=(\lil{x^i}{d})$ for all $i=1,\ldots,r$.\\
Note that $X=A\cup A_1\cup\cdots\cup A_d$, where $A$ contains $r$ points and each $A_j$ contains $r-1$ points, therefore $|X|=r+d(r-1)=n$ as required.

\begin{rmk}
The specific values of the coordinates of the points in $X$ are not crucial. It can be shown that if we construct $X$ and then move all the points very slightly, the theorem can be proved with only slight modifications. Therefore it is always possible to find a set $X$ that satisfies the theorem and is also in general position or even strong general position \cite{doignon}, \cite{perles}. On the other hand, we could allow $X$ to be a multiset, and let $A_j$ consist of $r-1$ copies of the same point $e^j$ (without any change to the proof below).
\end{rmk}

\section{Proof of Theorem 1}
Throughout the proof, assume $X$ is a set constructed as above, $\lil{X}{r}$ is a Tverberg partition of $X$, and $p=(\lil{p}{d})\in\bigcap_{i=1}^r conv(X_i)$ is a Tverberg point of $X$.

\begin{lem} 
$p=0$. In other words, $0$ is the only Tverberg point of $X$.
\end{lem}
\begin{proof}
Choose any $1\le j\le d$. There are exactly $r-1$ points $x\in X$ satisfying $x_j>0$ (these are the points in $A_j$). Therefore for some $1\le i\le r$, $X_i$ is contained in the closed half-space $\{x\in \RR^d:x_j\le 0\}$. Since $p\in conv(X_i)$, we must have $p_j\le 0$.

On the other hand, there are exactly $r-1$ points $x\in X$ satisfying $x_j<0$ (these are the points $x^i\in A$, for $i\neq i(j)$). Thus by a similar argument, $p_j\ge 0$. Therefore $p_j=0$ for all $1\le j\le d$.
\end{proof}

\begin{cor} 
For all $1\le i\le r$, $X_i\cap A$ contains exactly one point.
\end{cor}
\begin{proof}
Otherwise, there is some $1\le i\le r$ such that $X_i\cap A=\emptyset$ (since $A$ contains $r$ points). This implies that $p\in conv(X_i)\s conv(\cup_{j=1}^d A_j)$. Since every point $x\in\cup_{j=1}^d A_j$ satisfies $\sum_{j=1}^d x_j>0$, $p$ must also satisfy $\sum_{j=1}^d p_j>0$, but this contradicts Lemma 3.1.
\end{proof}

\begin{cor} 
Let $\lil{X}{r}$ be some Tverberg partition of $X$. We may always assume that $x^i\in X_i$ for all $1\le i\le r$ (otherwise we use Corollary 3.2 and renumber the sets of the partition). Then for every $1\le i\le r$ and every $1\le j\le d$:
$$|X_i\cap A_j|=\begin{cases} 0 & \text{if $i=i(j)$} \\1 & \text{otherwise.} \end{cases}$$
In particular, $|X_i|=a_i$, and $X$ satisfies the requirements of Theorem 1.
\end{cor}
\begin{proof}
Fix $1\le i\le r$ and $1\le j\le d$. From Lemma 3.1, we deduce that $0\in conv(X_i)$. Since $x^i$ is the only point in $X_i\cap A$, there is some $\lambda\in (0,1]$ and some $y^i\in conv(X_i\bs A)$, satisfying $$\lambda x^i+(1-\lambda)y^i=0.$$
Note that $\lambda$ cannot be $0$, because $0\notin conv(X\bs A)$.  If $i\neq i(j)$, we have $x^i_j<0$. Looking at the j-th coordinate of the above equation, we deduce that $y^i_j>0$. This is only possible if $X_i\cap A_j$ is non-empty (as only points from $A_j$ have positive values in the j-th coordinate). Thus
\be |X_i\cap A_j|\ge\begin{cases} 0 & \text{if $i=i(j)$} \\1 & \text{otherwise.}\end{cases} \ee
Since $A,\lil{A}{d}$ is a partition of $X$, we obtain (using eq. (1)):
\be |X_i|=1+\sum_{j=1}^d |X_i\cap A_j| \ge 1+(d-\#\{1\le j\le d:i(j)=i\})=a_i. \ee
However, $\lil{X}{r}$ is also a partition of $X$, thus $n=|X|=\sum_{i=1}^r |X_i|\ge \sum_{i=1}^r a_i=n$. This leads us to conclude all the inequalities in (2) and (3) are in fact equalities.
\end{proof}

\subsection*{The number of Tverberg partitions of $X$}

Sierksma conjectured that any set $X\s\RR^d$ of $n=T(d,r)$ points has at least $[(r-1)!]^d$ different Tverberg partitions (for known lower bounds, see \cite{vuvcic} and \cite{hell}). In the above construction, we can account for all Tverberg partitions of $X$, and their number is exactly $[(r-1)!]^d$. We state and prove this as the following proposition:

\begin{prop}
Let $d,r,n$ and $X$ be as above. Then $X$ has exactly $[(r-1)!]^d$ different Tverberg partitions.
\end{prop}
\begin{proof}
Assume $\lil{X}{r}$ is a Tverberg partition of $X$, such that $x^i\in X_i$ for all $1\le i\le r$ (we may always assume that $x^i\in X_i$, due to Corollary 3.2). For each $1\le j\le d$, define a bijection $\bij{j}$, satisfying $x\in X_{\si_j(x)}$ for all $x\in A_j$ ($\si_j$ is a bijection due to Corollary 3.3). Note that any Tverberg partition $\lil{X}{r}$ is uniquely determined by the $d$ bijections $\lil{\si}{d}$.

On the other hand, for any $d$ bijections $\bij{j}$, define for each $1\le i\le r$: $$X_i=\{x^i\}\cup\{\si_j^{-1}(i)\;:\;1\le j\le d, i\neq i(j)\}.$$ It is easy to verify that $0\in conv(X_i)$ for all $1\le i\le r$, therefore $\lil{X}{r}$ is a Tverberg partition of $X$, which is uniquely determined by the bijections $\lil{\si}{d}$. This shows that the number of Tverberg partitions of the set $X$ is equal to the number of unique choices of bijections $\bij{j}$ for $1\le j\le d$. For each $1\le j\le d$, $A_j$ contains $r-1$ points, so there are $(r-1)!$ possible choices for $\si_j$, and a total of $[(r-1)!]^d$ unique Tverberg partitions of $X$.
\end{proof}

\section*{Acknowledgments}
I am very grateful to my advisor Gil Kalai, for his guidance and support.\quad I would also like to thank Imre B\'{a}r\'{a}ny for a fruitful discussion.

\bibliographystyle{amsalpha}

\end{document}